\documentclass[psamsfonts]{amsart}
\usepackage{graphicx} % Required for inserting images
\usepackage{amsfonts}
\usepackage[utf8]{inputenc}
\usepackage{amsmath}
\usepackage{amssymb}

\usepackage{hyperref}
\usepackage{color}
\usepackage{lipsum}
\usepackage[title]{appendix}
\usepackage{mathtools}
\usepackage{ulem}
\usepackage{tikz}

\usepackage[inline]{asymptote}
\usepackage{lmodern}
\usepackage[paperheight=10.75in,paperwidth=8.25in,margin=1in,heightrounded]{geometry}

\usepackage{graphicx}

\def\R{\mathbb{R}}

\def\N{\mathbb{N}}

\def\E{\mathbb{E} }

\def\al{\alpha}
\def\ep{\varepsilon}

\def\v{{\mathbf{v}}}
\def\t{{\mathbf{t}}}
\def\n{{\mathbf{n}}}

\newtheorem{theorem}{{Theorem}}[section]
\newtheorem{proposition}[theorem]{{Proposition}}%[section]
\newtheorem{isom.ext}[theorem]{{Trivial isometric extension}}%[section]
\newtheorem{definition}[theorem]{{Definition}}%[section]
\newtheorem{lemma}[theorem]{{Lemma}}%[section]
\newtheorem{corollary}[theorem]{{Corollary}}%[section]
\newtheorem{remark}[theorem]{{Remark}}%[section]
\newtheorem{notation}[theorem]{{Notation}}%[section]
\definecolor{purple}{rgb}{0.65,0.12,0.94}
\definecolor{forestgreen}{rgb}{0.4,0.64,0.13}
\title{$C^{1}$-isometric embeddings of Riemannian spaces in Lorentzian spaces}
\author[Alaa Boukholkhal]{Alaa Boukholkhal}\thanks{I am thankful to the reviewer for their valuable comments and remarks. \\
The author is supported by the LABEX MILYON (ANR-10-LABX-0070) of
Université de Lyon, within the program "Investissements d'Avenir" (ANR-11-IDEX-0007) operated by the French National Research Agency (ANR).
}
\address{UMPA, École normale supérieure de Lyon, France}
\email{mohamed.boukholkhal@ens-lyon.fr}
\begin{document}
\begin{center}
\begin{abstract} 
  For any compact Riemannian manifold $(V,g)$ and any Lorentzian manifold $(W,h)$, we prove that any spacelike embedding $f: V \rightarrow W$  that is long ($g\leq f^{*}h$) can be $C^{0}$-approximated by a $C^{1}$  isometric embedding $F: (V,g) \rightarrow (W,h)$.
\end{abstract}
\end{center}
  
\maketitle

\section{Introduction}
The theorem of Nash-Kuiper \cite{nash1954c1},\cite{kuiper1955c1} states that any Riemannian manifold $(V,g)$ admits a $C^{1}$ isometric embedding into the Euclidean space $(\E^{q},g_{euc})$ as long as there exists an embedding $f:V\rightarrow \E^{q}$, such that $g-f^{*}g_{euc}$ is positive semi-definite (we say that $f$ is short). This theorem was one of the main inspirations of Gromov's $h$-principle theory (\cite{gromov1986}, \cite{eliash}), and it was generalized in many contexts (symplectic \cite{Dambrasym}, contact \cite{Dambracont}, sub-Riemannian \cite{LeDonne}, for totally real embeddings \cite{theilliere2022convex}). In the pseudo-Riemannian setting, the theorem extends immediately if there is enough codimension:  
\begin{theorem}[Gromov \cite{gromov1986} p204, \cite{GromovNash} p34]
Let $(V,g)$ and $(W,h)$ be two pseudo-Riemannian manifolds, where $(r_{+}, r_{-})$ $($respectively $(q_{+},q_{-})$$)$ is the signature of $g$ $($respectively $h$$)$. Let $f: V \rightarrow W$ be an embedding such that $f^{*}h$ has the same signature as $g$. 
If $q_{+} > r_{+}$ and   $q_{-}>r_{-}$, then $f$ can be $C^{0}$-approximated by a $C^{1}$ isometric embedding $F: (V,g) \rightarrow (W,h)$. 
    \end{theorem}
  In the classical Nash-Kuiper approach, a normal vector field is used to deform the initial embedding and increase the metric. In the setting of Gromov's theorem, we have both timelike and spacelike normal directions, which allows deforming the embedding in both directions and hence, increasing and decreasing the metric following which direction is used. In the following, we will be interested in the codimension $1$ case which is not covered in the theorem above. More precisely, we will study the case where $q_{-}>r_{-}=0$, and $r_{+}=q_{+}$: 
  \begin{theorem}\label{T}
Let $(V,g)$ be a compact Riemannian manifold, $(W,h)$ a  pseudo-Riemannian manifold and $f:V \rightarrow (W,h)$ a long spacelike embedding  i.e, $g\leq f^{*}h$. Then, the embedding $f$ can be $C^{0}$-approximated by a $C^{1}$ isometric embedding $F: (V,g) \rightarrow (W,h)$.
\end{theorem}

The main difference here with the classical Nash-Kuiper theorem and Gromov's generalization is that the normal at any point of $V$ (seen as an embedded hypersurface in $W$) is timelike. We will use this normal to deform the initial embedding $f$ following the Nash-Kuiper process. However, we need to ensure that the tangent space is far from the light cone at each step of the process to guarantee the $C^{1}$ regularity at the end (see lemma \ref{lem2}).

As in the classical Nash-Kuiper theorem, Theorem \ref{T} reflects high flexibility when considering $C^{1}$-isometric embeddings, which is not the case for higher regularity. Indeed, for any surface $\Sigma$ of genus $\geq 2$ and any cocompact lattice $\Gamma$ in $SO^{\circ}(2,1)$ isomorphic to $\pi_1(\Sigma)$, there exists a unique hyperbolic metric on $\Sigma$ that admits a smooth isometric embedding in the quotient of the $(2+1)$-dimensional solid timelike cone by $\Gamma$. On the other hand, Theorem \ref{T} implies: 
\begin{corollary}
  Let $\Sigma$ and $\Gamma$ as above. Then, any hyperbolic metric on $\Sigma$ admits a $C^{1}$-isometric embedding in the quotient of the $(2+1)$-dimensional solid timelike cone by $\Gamma$.
\end{corollary}

\section{The Nash-Kuiper process} 
We will give the proof of Theorem \ref{T} in the case where $(V,g)$ is a compact Riemannian surface and $(W,h)$ is a $3$-dimensional Lorentzian manifold. The methods used here can be applied in higher dimensions without employing any additional ideas. 
  \medskip 
 
  In order to apply the Nash-Kuiper process in this setting, we need to start with a long embedding, that is an embedding $f: V\rightarrow W$, such that the metric $\Delta:=f^{*}h-g $ is positive semi-definite. We call $\Delta$ the isometric default.
  \medskip 

 The Nash-Kuiper process consists of reducing the isometric default gradually. More precisely, we will consider a sequence $(g_n)_{n\in \N}$ of Riemannian metrics defined by $g_n=g+\delta_n\Delta$, where $(\delta_n)_{n\in\N}$ is a decreasing sequence of positive numbers converging to $0$. Using a variant of the convex integration formula called the corrugation process \cite{theilliere2022convex}, we build a sequence $(f_n)_{n\in \N}$ of embeddings from $V$ to $W$ such that: 
\begin{enumerate}
    \item $\parallel f_{n}^{*}h-g_n \parallel < \parallel g_{n}-g_{n+1} \parallel $,
    \item $f_n$ is $C^{0}$ close to $f_{n-1}$,
    \item $\parallel f_n - f_{n+1}\parallel_{g,\widetilde{h}}$ is under control,
   
\end{enumerate}
  where $\parallel . \parallel_{g,\widetilde{h}}$ is the operator norm with respect to the metric $g$ on $V$ and any fixed Riemannian metric on $W$, we denote it by $\widetilde{h}$. Note that the first condition implies that $f_n$ is long for the metric $g_{n+1}$ since $g_n \ge g_{n+1}$. Therefore, we need to be able, starting from a long embedding, to construct an embedding which is $\ep$-isometric, i.e, with small isometric default. To build $f_{n+1}$ from $f_n$, we will use Theillière's corrugation process formula:
  
\begin{definition}[\cite{theilliere2022convex}]\label{defcorrug}
    Let $f:U \rightarrow (W,h)$ be a map from an open set $U\subset V$, $\pi : U \rightarrow \mathbb{R}$ a submersion, $N\in \N^{*}$ and $\gamma : U \times \mathbb{R}/\mathbb{Z} \rightarrow f^*TW$ be a smooth loop family such that $\gamma(p,.) : \mathbb{R}/\mathbb{Z} \rightarrow f^*TW_p$ for every $p\in U$. The map defined by Corrugation Process is given by: 
    $$ F : p \mapsto exp_{f(p)} (\frac{1}{N} \Gamma(p,N\pi(p)))$$
    where $\Gamma(p,s)=\int_0^s (\gamma(p,t)-\Bar{\gamma}(p)dt $ and $\Bar{\gamma}(p)=\int_0^1 \gamma(p,t) dt$. We call $N$ the corrugation number.
\end{definition}
\begin{remark}\label{loccor}
    Unlike the convex integration formula, the corrugation process formula is defined directly on the manifold (coordinate free). Moreover, $F(p)$ is determined by the value of $f, \, \gamma,\, \pi$ at $p$.
\end{remark}
In addition, since we are using a Riemannian metric $\widetilde{h}$ on $W$ to compute the $C^{1}$ distance, it will be important to control the norm with respect to $\widetilde{h}$ of the timelike normal to $f_n$ for each step $n$. We will prove that this is possible and that the evolution of this norm at each $n$ can be absorbed by a good choice of the sequence $\delta_n$, which will ensure the $C^{1}$ convergence (see section $4$).  
\section{The fundamental example}

As stated above, using the corrugation process, we can deform the embedding $f$ on any fixed open set. Since $V$ is compact, we will use this fact, to deform $f$ on each chart. Hence, the problem is reduced to a local one. Recall that, the fundamental step in the Nash-Kuiper process is to construct $\varepsilon$-isometric maps. For simplicity, we will consider the specific case, where $(W,h)$ is the Minkowski space $\R^{2,1}$ of dimension three and $f: C:=[0,1]^{2}\rightarrow W$ is a spacelike embedding ($f^{*}h$ is a Riemannian metric). The general case follows immediately (see section $5$).

\subsection{The case of primitive metrics}
We consider here the case of primitive metrics (metrics of the form $\mu :=f^{*}h - \eta \, d\ell \otimes d\ell$), the general case will follow by iteration and it will be presented just after. 

    \begin{proposition}\label{local}
    Let $(W, h)$ be the Minkowski space $\mathbb{R}^{2,1}$ with the usual Lorentzian metric and let $f : C = [0,1]^2 \rightarrow (W,h)$ be a smooth spacelike embedding. Consider the metric $\mu = f^{*}h - \eta \, d\ell \otimes d\ell$, where $\ell : C \rightarrow \mathbb{R}$ is an affine projection and $\eta : C \rightarrow \mathbb{R}_{\geq 0}$ is a smooth function. Then, for any $\varepsilon > 0$, there exists a smooth $\varepsilon$-isometric spacelike embedding $F : (C, \mu) \rightarrow (W, h)$, i.e. an embedding such that
    $$
    \|F^{*}h - \mu\|_{C^{0}} \leq \varepsilon.
    $$
    Moreover, $F$ is $C^{0}$ close to $f$.
\end{proposition}

\medskip
\begin{remark}
    The proposition above can be found in \cite{Boukh}. The map $F$ constructed in the proof will be used in all what follows, we therefore give a proof for completeness. 
\end{remark}
\begin{proof}
     The corrugation process formula in this case is: 
$$
\forall p \in C, \quad F(p)=f(p)+\frac{1}{N} \int_{0}^{N \ell(p)}(\gamma(p, s)-\bar{\gamma}(p)) \mathrm{d} s
$$
We will now construct the family of loops $\gamma: C \times(\mathbb{R} / \mathbb{Z}) \rightarrow \mathbb{R}^{2,1}$: 
\medskip
\\ Let $(v,u)$ be an orthonormal basis for the metric $f^*h$ on $C$ such that $v\in ker(d\ell)$, and let 
$$\mathbf{v}=df(v) \; \; \; \; \mathbf{t}=df(u)$$
and $\mathbf{n}$ to be a normal field to $f$ with respect to the metric $h$. Note that $\mathbf{n}$ is a timelike vector field of norm $-1$ since $f$ is spacelike. The basis $\{\v,\t,\n\}$ is called $the$ $corrugation$ $frame$.
\\ We can now define the loop family $\gamma$ by: 

\begin{equation}\label{loop}
    \gamma(\cdot, s):=r(\cosh(\theta) \t+ \sinh (\theta)  \n) \quad \text {with} \quad \theta=\alpha \cos (2 \pi s)
\end{equation}
and where $r$ and $\alpha$ are functions on $C$ that will be chosen below such that: 
\begin{equation}\label{average}
      \bar{\gamma} = \frac{\t}{d\ell(u)}
\end{equation}
where $\bar{\gamma}$ is the average 
$$  \bar{\gamma} =r\left(\int_{0}^{1} \cosh(\alpha \cos (2 \pi s)) \mathrm{d} s\right) \t $$
We call $\alpha$, \textbf{the amplitude factor}, see Figure \ref{dilation}.

\begin{figure}[h!]
    \includegraphics[width=0.6\textwidth]{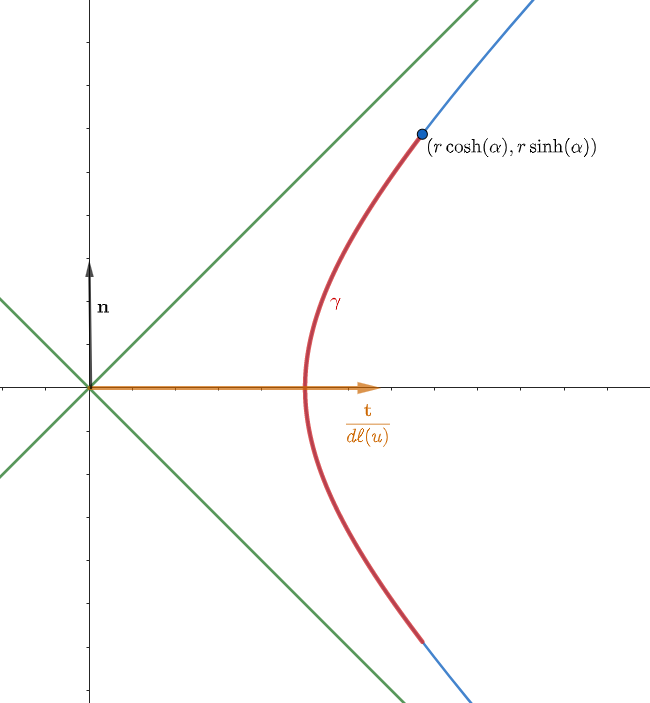}
    
    \centering
    \caption{The support of $\gamma$ (red curve) lies in the ball of radius $r$ for the Lorentzian metric $h$ and its average is $\frac{\t}{d\ell(u)}$.}
    \label{dilation}
\end{figure}

The differential of $F$ at the point $p$ is given by: 
$$
    dF  = df + (\gamma(p, N \ell(p))-\bar{\gamma}(p)) \otimes d \ell+\frac{1}{N} \int_{0}^{N \ell(p)} d(\gamma(p, s)-\bar{\gamma}(p)) \mathrm{d} s .$$
Let $$
L:=d f+(\gamma(\cdot, N \ell)-\bar{\gamma}) \otimes d\ell
$$
We call $L$ the target differential and we denote it by $L=L(f, \gamma, N, \ell)$. 
This target differential coincides with $d f$ on ker $d \ell$ and differs from it on the transversal directions by the addition of a term depending on $\gamma$. It is the term that will participate the most in the deformation. Indeed, since $\gamma$ is $1$-periodic on its second variable $s$, we have $$ \int_{0}^{N \ell(p)}(\gamma(p, s)-\bar{\gamma}(p)) \mathrm{d}s= \int_{\lfloor N\ell(p) \rfloor}^{N \ell(p)}(\gamma(p, s)-\bar{\gamma}(p)) \mathrm{d}s $$ and $\gamma$ is continuous on a compact set, we get 
\begin{equation}\label{target}
    F^*h = L^*h + O(\frac{1}{N}) 
\end{equation}
If $N$ is large enough, then $F$ induces a Riemannian metric on $C$ if $L$ does. We will now choose $r$ and $\alpha$ so that $L$ is $\mu$-isometric $(L^*h=\mu)$, this will imply for $N$ large enough that $F$ is $\varepsilon$-isometric for the metric $\mu$. We have:
\begin{equation}\label{targetdiff}
    L= df + (r(\cosh(\theta) \t+ \sinh (\theta)\n) - \frac{\t}{d\ell(u)}) \otimes d\ell
\end{equation}
and hence 
$$L^*h = f^*h + (r^{2} - \frac{1}{d\ell(u)^{2}})d\ell\otimes d\ell$$
We choose now $r$ so that $$r^{2} - \frac{1}{d\ell(u)^{2}} = - \eta $$
this is of course possible since we assumed $\mu$ to be Riemannian ($\mu(u)=1-\eta d\ell(u)^{2} > 0$). 
\medskip

 The only thing left now is to choose $\alpha$, so that equation \ref{average} is satisfied. We can write equation \ref{average} as:
$$\varphi(\alpha)=\frac{1}{rd\ell(u)}$$
where $\varphi$ is the map defined by  \begin{equation}\label{fi}
    \varphi: \alpha \in [0,+\infty[ \mapsto \int_{0}^{1} \cosh(\alpha \cos (2 \pi s)) \mathrm{d} s \in [1,+\infty[
\end{equation}
We conclude that equation \ref{average} has a unique, smooth solution over $C$, since $\varphi$ is smooth, increasing and surjective (see Figure \ref{graph}) and $$ \frac{1}{rd\ell(u)} \ge 1$$
\begin{figure}[h]
    \includegraphics[width=0.6\textwidth]{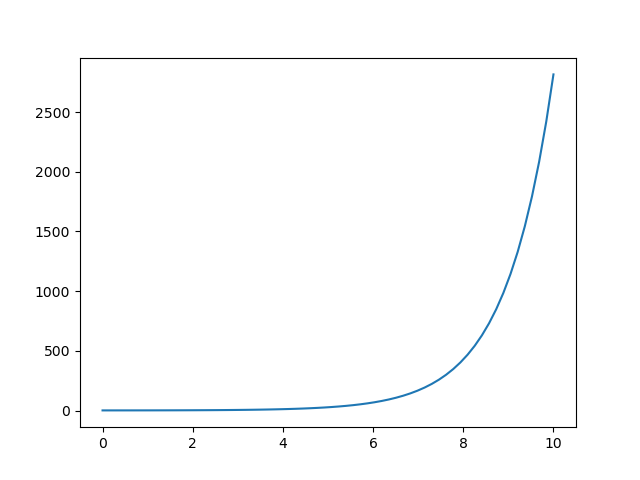}
\centering
\caption{The graph of the map $\varphi$}
    \label{graph}
\end{figure}
\end{proof}
\begin{notation}
    The map $F$ constructed in the proof above using the corrugation process with the family of loops \ref{loop} is denoted by:
    $$CP(f,\mu,N)$$ 
    Proposition \ref{local} shows that the map $CP(f,\mu,N)$ is $\varepsilon$-isometric if $N$ is large enough. 
\end{notation}
   \begin{remark}\label{gluing}
    Notice that for any $p\in C$ such that $\eta(p)=0$ $(\mu(p)=f^{*}h (p))$, we have $f(p)=F(p)$. This fact will help us go from the local construction to the global one (see section $5$).
\end{remark}
It would be important in the following to understand the behavior of the amplitude factor with respect to the metric. 
\begin{lemma}\label{behavior}
    The amplitude factor $\al$ decreases when the isometric default decreases.
    
\end{lemma}
\begin{proof}
    Notice that in the proof of proposition \ref{local}, we have 
\begin{align*}
    \al=& \varphi^{-1}\left(\frac{1}{rd\ell(u)}\right) \\
    =& \varphi^{-1}\left(\frac{1}{\sqrt{1-\eta d\ell(u)^{2}}}\right)
\end{align*}
Since $\varphi^{-1}$ is increasing, we can conclude. 
\end{proof}
\medskip

  Fix now a Riemannian metric $\widetilde{h}$ on $W$, we will denote  by $\parallel . \parallel_{\widetilde{h}} $ the norm with respect to $\widetilde{h}$. By taking $N$ large enough, the maps $F=CP(f,\mu,N)$ and $f$ can be arbitrarily close with respect to $\widetilde{h}$. 
  Controlling the difference $dF - df$ is more delicate.

\begin{lemma}\label{lem}
   For $f$ and $F=CP(f,\mu,N)$ as before, we have the following inequality: 
    $$\parallel dF-df \parallel_{g,\widetilde{h}} \leq O(1/N) + M(\alpha_{max})\eta^{\frac{1}{2}} \parallel d\ell \parallel_{g,\widetilde{h}}    (\parallel df \parallel_{g,\widetilde{h}}  +\parallel \n \parallel_{\widetilde{h}})  $$
    where by $\parallel . \parallel_{g,\widetilde{h}}$ we mean the operator norm with respect to any metric $g\leq f^{*}h$ and $\widetilde{h}$ and $M(\alpha_{max})$ is a constant that depends on the maximal value $\alpha_{max}$ of the amplitude factor $\alpha$ over $C$.
\end{lemma}
\begin{proof}
    Since $\parallel dF - L \parallel_{g,\widetilde{h}} = O(1/N) $, the inequality reduces to the study of $\parallel df - L \parallel_{g,\widetilde{h}}$. 
     Since we have $L(v)=df(v)$, the difference $L-df$ reduces to studying $(L-df)(u)$, we have:
     \begin{align*}
             \parallel (L-df)(u) \parallel_{\widetilde{h}} =&  \parallel (rd\ell(u)\cosh(\theta) - 1)\t + rd\ell(u)\sinh(\theta)\n \parallel_{\widetilde{h}}\\
    \leq &  \parallel (rd\ell(u)\cosh(\theta) - 1)\t \parallel_{\widetilde{h}} + \parallel  rd\ell(u)\sinh(\theta)\n \parallel_{\widetilde{h}} \\
    \end{align*}
Using the fact that $\varphi(\al)=\frac{1}{rd\ell(u)}$, the inequality becomes: 
    
     \begin{align*}
     \parallel (L-df)(u) \parallel_{\widetilde{h}} \leq &\left( \left| \frac{\cosh(\theta)-\varphi(\al) }{\varphi(\alpha) } \right| + \left| \frac{\sinh(\theta)}{\varphi(\alpha)}\right| \right)  ( \parallel \t \parallel_{\widetilde{h}} + \parallel \n \parallel_{\widetilde{h}}) \\
    \leq& \left( \frac{\sqrt{(\cosh(\theta) -\varphi(\al))^{2}} +\sinh(\al) }{\varphi(\alpha) }  \right)  ( \parallel \t \parallel_{\widetilde{h}} + \parallel \n \parallel_{\widetilde{h}}) \\
    \leq& \left( \frac{\sqrt{\cosh(\theta)^{2} +\varphi(\al)^{2}-2\cosh(\theta)\varphi(\al)} +\sinh(\al) }{\varphi(\alpha) }  \right)  ( \parallel \t \parallel_{\widetilde{h}} + \parallel \n \parallel_{\widetilde{h}}) \\
     \leq& \left( \frac{\sqrt{2\cosh(\al)^{2} -2\varphi(\al)} +\sinh(\al) }{\varphi(\alpha) }  \right)  ( \parallel \t \parallel_{\widetilde{h}} + \parallel \n \parallel_{\widetilde{h}}) 
     \end{align*}
For the last inequality we used the fact that $\theta = \alpha cos(2\pi s)$ and $\cosh(\al) \geq \varphi(\al) $. Recall now that $r^{2} - \frac{1}{d\ell(u)^{2}} = - \eta$, hence 
\begin{equation}\label{equality}
    \eta d\ell(u)^{2}=1-(rd\ell(u))^{2}= 1-\frac{1}{\varphi(\al)^{2}} 
\end{equation}
The function:
$$\alpha \in [0,+\infty[ \mapsto \frac{ \frac{(\sqrt{2\cosh(\al)^{2} -2\varphi(\al)} +\sinh(\al) )^{2}}{\varphi(\alpha)^{2 }  }}{1-\frac{1}{\varphi(\alpha)^{2}}}$$ is continuous on $]0,\infty[$ and well defined at $0$ (by L'Hôpital's rule, see the appendix). Therefore, as long as $\alpha$ is bounded from above (say by $\alpha_{max})$, which is the case here, there exists $M(\alpha_{max})$ verifying: 
$$\frac{\sqrt{2 \cosh(\al)^{2}-2\varphi(\al)} +\sinh(\al)} {\varphi(\alpha) } \leq M(\alpha_{max})\sqrt{ 1-\frac{1}{\varphi(\alpha)^{2}}}$$ 
By equation \ref{equality}, we get 
$$\parallel (dF-df)(u) \parallel_{\widetilde{h}} \leq O(1/N) + M(\alpha_{max}) \eta^{\frac{1}{2}} \mid d\ell(u)\mid   ( \parallel \t \parallel_{\widetilde{h}} + \parallel \n \parallel_{\widetilde{h}})  $$
and since $$\mid d\ell(u)\mid \leq \parallel d\ell \parallel_{g,\widetilde{h}}\; \parallel u \parallel_{g},$$ 
$$\parallel \t \parallel_{\widetilde{h}}=\parallel df(u) \parallel_{\widetilde{h}}\leq \parallel df \parallel_{g,\widetilde{h}}\; \parallel u \parallel_{g}$$ and $\parallel u \parallel_{g} \leq \parallel u \parallel_{f^{*}h}= 1$
we conclude that
$$\parallel dF-df \parallel_{g,\widetilde{h}} \leq O(1/N) + M(\alpha_{max}) \eta^{\frac{1}{2}} \parallel d\ell \parallel_{g,\widetilde{h}}   ( \parallel df \parallel_{g,\widetilde{h}}   + \parallel \n \parallel_{\widetilde{h}})  $$
  \end{proof} 
Before considering the general case, we need to establish the following lemma:
 \begin{lemma}\label{lem2}
     Let $f$ and $F=CP(f,\mu,N)$ as before and let $K(\al_{max})=2\cosh(\al_{max})+1$, where $\alpha_{max}$ is the maximal value of the amplitude factor $\alpha$ over $C$. Then, for $N$ large enough, we have
      $$ \parallel dF \parallel_{g,\widetilde{h}} \leq K(\alpha_{max})  ( \parallel df \parallel_{g,\widetilde{h}} + \parallel \n \parallel_{\widetilde{h}})$$
         Moreover, if we denote by $\n_F$ a normal vector field of $F(C)$ with respect to the Lorentzian metric 
         $h$ of norm $-1$, then 
          $$\parallel \n_F \parallel_{\widetilde{h}} \leq K(\alpha_{max})  ( \parallel df \parallel_{g,\widetilde{h}} + \parallel \n \parallel_{\widetilde{h}})$$
  \end{lemma}
\begin{proof}
   From equation \eqref{targetdiff}, we have
    \begin{align*}
        \parallel L(u) \parallel_{\widetilde{h}} = & \parallel  rd\ell(u)\left[ \cosh(\theta) \t+ \sinh (\theta)\n\right]\parallel_{\widetilde{h}} \\ 
        \leq &  \frac{ \cosh(\theta) + |\sinh (\theta)|}{\varphi(\al)}( \parallel \t \parallel_{\widetilde{h}} + \parallel \n \parallel_{\widetilde{h}})\\
        \leq& \frac{2\cosh{\al}}{\varphi(\al)}( \parallel df \parallel_{g,\widetilde{h}} + \parallel \n \parallel_{\widetilde{h}})\\
    \end{align*}
    where for the last inequality, we used the fact
    $$\parallel \t \parallel_{\widetilde{h}}=\parallel df(u) \parallel_{\widetilde{h}}\leq \parallel df \parallel_{g,\widetilde{h}} \parallel u \parallel_{g}$$ 
    and that $\parallel u \parallel_{g} \leq \parallel u \parallel_{f^{*}h}= 1$. Now, since $\eta$ is bounded and $\eta d\ell(u)^{2}= 1-\frac{1}{\varphi(\al)^{2}}$ ($\alpha$ depends continuously on $\eta$), we can restrict $\al$ to a closed interval $[0,\alpha_{max}]$ and we can take $K_1 = \max \frac{2\\cosh{\al}}{\varphi(\al)}  $. Since $\parallel dF - L \parallel_{g,\widetilde{h}} = O(1/N) $ and $L(v)=df(v)$, we conclude that 
    $$ \parallel dF \parallel_{g,\widetilde{h}} \leq O(1/N) + K_1( \parallel df \parallel_{g,\widetilde{h}} + \parallel \n \parallel_{\widetilde{h}}) $$
    and for $N$ large enough, we get 
    $$ \parallel dF \parallel_{g,\widetilde{h}} \leq (K_1 +1)( \parallel df \parallel_{g,\widetilde{h}} + \parallel \n \parallel_{\widetilde{h}}) $$

For the second part of the lemma, we consider for each point $p$, the vector $\n_L= -\sinh (\theta) \t+ \cosh(\theta)\n$ normal to $Im(L)$ (at the point $p$) with respect to $h$. We have:
$$h(\n_L,\n_L)= -1$$
$$h(\n_L,dF(v))=O(1/N) + h(\n_L,L(v))=  O(1/N)$$
$$h(\n_L,dF(u))=O(1/N) + h(\n_L,L(u))=  O(1/N)$$
Hence, for $N$ large enough, and by continuity of $h$ and compactness of $C$, we have for every $p\in C$:
$$\parallel \n_F - \n_L \parallel_{\widetilde{h}} \leq O(1/N)$$
Therefore, since $$\parallel \n_L \parallel_{\widetilde{h}} \leq \cosh(\al) ( \parallel \t \parallel_{\widetilde{h}} + \parallel \n \parallel_{\widetilde{h}})$$
and by the same argument as before 
$$\parallel \n_L \parallel_{\widetilde{h}} \leq \cosh(\al) ( \parallel df \parallel_{g,\widetilde{h}} + \parallel \n \parallel_{\widetilde{h}})$$
we get for $K_2 = \max_{[0,\alpha_{max}]} \cosh(\al)+1 $, and $N$ large enough
\begin{align*}
    \parallel \n_F \parallel_{\widetilde{h}} \leq& K_2 ( \parallel df \parallel_{g,\widetilde{h}} + \parallel \n \parallel_{\widetilde{h}}).
\end{align*}
We can conclude by taking $K(\al_{max})=2\cosh(\al_{max})+1$ since $K(\al_{max})\geq \max \{ K_1,K_2\}$ ($K_1 \leq  \max_{[0,\alpha_{max}]} 2\cosh(\al)$ since $\varphi(\al) \geq 1$).
\end{proof}

\subsection{The case of a general metric}
We can now generalize the previous construction for any metric. Let $f:C\rightarrow \R^{2,1}$ be a smooth long spacelike embedding, and let $\mu$ be a Riemannian metric on $C$ of the form $f^{*}h - \sum_{i=1}^{k} \eta_i d\ell_{i}^{2}$, where each $\eta_i$ is a smooth positive function on $C$ and each $\ell_i : C \rightarrow \R$ is an affine projection (we can always write $\mu$ in this form see section $5$). We can construct a $\varepsilon$-isometric embedding $F: C\rightarrow \R^{2,1}$ for the metric $\mu$ using a successive corrugation process. More precisely, we construct a sequence of maps $f=F_0, F_1,..., F_k=F $, where 
\begin{equation*}\label{generalconstr}
    F_j=CP(F_{j-1}, \mu_j, N_{j})
\end{equation*}
where $\mu_j = F_{j-1}^{*}h - \eta_{j} d\ell_{j} \otimes d\ell_{j}$. By taking $N_j$ large enough, $F_j$ will be long for the metric $\mu_{j+1}$, and hence, we can use proposition \ref{local} to iterate the process.
\medskip

  To control the difference $dF-df$ now, we need to use lemmas \ref{lem} and \ref{lem2} successively. 
  
 \begin{proposition}\label{prop}
 For any $\varepsilon \geq 0$, we have the following inequality: 
  \begin{align*}
            \parallel dF - df \parallel_{g,\widetilde{h}} \leq \varepsilon+ M(\al_{max}) c \parallel \Delta \parallel_{g,\widetilde{h}}  ( \parallel df \parallel_{g,\widetilde{h}} + \parallel \n \parallel_{\widetilde{h}}) \widetilde{K}(\al_{max}) 
          \end{align*}
          where by $\parallel . \parallel_{g,\widetilde{h}}$ we mean the operator norm with respect to any metric $g\leq f^{*}h$ and $\widetilde{h}$. The constant $c$ depends only on $d\ell_1,...,d\ell_k$, $M(\al_{max})$ depends on the maximal value $\alpha_{max}$ of the amplitude factors $\alpha^{j}$ of $F_j=CP(F_{j-1}, \mu_j, N_{j})$ over $C$ for all $j\in [\![ 1,k ]\!]$ and $\widetilde{K}(\al_{max})=2^k(2\cosh(\al_{max})+1)^{k}$.
 \end{proposition}
 \begin{proof}
 We have by lemma \ref{lem}:
     \begin{align*}
     \parallel dF - df \parallel_{g,\widetilde{h}} \leq& \sum_{j=0}^{k-1} \parallel dF_{j+1} - dF_j \parallel_{g,\widetilde{h}} \\
     \leq& \sum_{j=0}^{k-1} \left[O(1/N_j) + M(\alpha^{j}_{max}) \eta_j^{\frac{1}{2}} \parallel d\ell_j \parallel_{g,\widetilde{h}} ( \parallel \n_j \parallel_{\widetilde{h}} + \parallel dF_j \parallel_{g,\widetilde{h}})\right] \\
      \end{align*} 
      Here, $\n_j$ is a normal to $TF_j$, $\alpha^{j}_{max}$ is the maximum value of the amplitude factor for $F_j=CP(F_{j-1}, \mu_j, N_{j})$ over $C$ and $N_j$ is taken large enough as in lemma \ref{lem2}. Take now $\alpha_{max}=\max \alpha^{j}_{max}$ and let $K= 2\cosh(\alpha_{max}) +1$. Note that $K$ satisfy
      $$K \geq K_j(\alpha^{j}_{max})$$ 
      where $K_j(\alpha^{j}_{max})=2\cosh(\alpha^{j}_{max})+1$ as in lemma \ref{lem2}. We have 
      $$
          \parallel dF - df \parallel_{g,\widetilde{h}} \leq \sum_{j=0}^{k-1} \left[ O(1/N_j) + M(\alpha_{max}) \eta_j^{\frac{1}{2}} \parallel d\ell_j \parallel_{g,\widetilde{h}} (2K)^{j} (  \parallel df \parallel_{g,\widetilde{h}} + \parallel \n_0 \parallel_{\widetilde{h}})\right]$$
          Recall that we have $\Delta= \sum_{j=1}^{k} \eta_j d\ell_j^{2}$, then there exists a constant $c>0$ that depends only on $d\ell_1,...,d\ell_k$ and by taking each $N_j$ large enough we get
          \begin{align*}
            \parallel dF - df \parallel_{g,\widetilde{h}}  \leq \varepsilon + M(\alpha_{max}) c \parallel \Delta \parallel^{\frac{1}{2}}_{g,\widetilde{h}}   (  \parallel df \parallel_{g,\widetilde{h}} + \parallel \n_0 \parallel_{\widetilde{h}}) \widetilde{K}(\alpha_{max})
          \end{align*}
where $\widetilde{K}(\alpha_{max})=2^k (2\cosh(\al_{max})+1)^{k}$.
 \end{proof}

\section{The $C^{1}$ embedding}

We are now ready to apply the Nash-Kuiper process, recall that $f: (C,g) \rightarrow \R^{2,1} $ is a long spacelike embedding, meaning that the isometric default $\Delta=f^{*}h-g=:\sum_{1}^{k}\eta_j d\ell_j^{2}$ is positive semi-definite (we recall that $h$ is the Lorentzian metric on the Minkowski space). The goal is to construct a $C^{1}$ isometric embedding $f_\infty : (C,g) \rightarrow (\R^{2,1},h) $. The construction goes as follows: 
\begin{itemize}
    \item First, notice that by lemma \ref{behavior}, if the isometric default converges to $0$, then the amplitude  factor $\alpha$ converges to $0$ too. In the Nash-Kuiper process (see section $2$), the isometric default is getting smaller in each step. Hence, we can choose $\alpha_{max}$ from the beginning such that the amplitude factor is bounded above by $\alpha_{max}$ for any $n$. 
\item Up to taking large corrugation numbers at each step and since $C$ is compact, we can choose the $1$-forms $d\ell_1,...,d\ell_k$ once for all (see lemma $29.3.1$ in \cite{eliash}).
    \item  We consider a sequence $(g_n)_{n\in \N}$ of Riemannian metrics defined by $g_n=g+\delta_n\Delta$, where $(\delta_n)_{n\in\N}$ is a decreasing sequence of positive numbers converging to $0$ such that: 
    $$\sum \sqrt{\delta_{n-1}-\delta_{n}}(2\widetilde{K})^{n}< \infty$$
    where $\widetilde{K}= 2^{k}(2\cosh(\al_{max}) +1 )^{k}$, where $k$ is the number of $1$-forms $d\ell_1,...,d\ell_k$ in the previous point.
    \item For any fixed $\varepsilon >0$, we introduce a sequence $a_n$ of positive numbers such that $$ \sum a_n < \varepsilon$$
    Starting from $f_0$, we build by corrugation process a sequence $(f_n)_{n\in\N}$ of embeddings such that: 
    \begin{enumerate}

           \item $\parallel f_{n}^{*}h-g_n \parallel_{g,\widetilde{h}} \leq \parallel g_{n+1}-g_n \parallel_{g,\widetilde{h}} $,
    \item $\parallel f_n-f_{n-1}\parallel_{\widetilde{h}} \leq a_n$,
    \item $\parallel df_n-df_{n-1}\parallel_{\widetilde{h}} \leq a_n + T \parallel g_n - g_{n-1}\parallel^{\frac{1}{2}}_{g,\widetilde{h}} (2\widetilde{K})^{n}\;\,$,
\end{enumerate}

where $T$ is a constant that does not depend on $n$. By the construction \ref{generalconstr}, we only need to take large corrugation numbers at each step to satisfy the first two conditions. Regarding the third condition, notice that condition one ensures that $f_{n-1}$ is long for the metric $g_{n}$, applying \ref{generalconstr} on $f_{n-1}$ we can build an almost isometric embedding $f_n$ for the metric  $g_{n}$. We have by  proposition \ref{prop}:
      \begin{align*}
            \parallel df_{n} - df_{n-1} \parallel_{\widetilde{h}}  \leq a_n + M(\alpha_{max}) c \parallel g_n - f_{n-1}^{*}h \parallel^{\frac{1}{2}}_{g,\widetilde{h}}  \widetilde{K} ( \parallel df_{n-1}\parallel_{g,\widetilde{h}} +\parallel \n_{f_{n-1}} \parallel_{\widetilde{h}} )
          \end{align*}
          where by $\n_{f_{n-1}}$ we mean the normal to $f_{n-1}(C)$ with respect to the Lorentzian metric $h$. Using lemma \ref{lem2} successively, we get 
          \begin{align*}
            \parallel df_{n} - df_{n-1} \parallel_{\widetilde{h}}  \leq& a_n +  M(\alpha_{max}) c \parallel g_n - f_{n-1}^{*}h \parallel^{\frac{1}{2}}_{g,\widetilde{h}}  (2\widetilde{K})^{n} ( \parallel df_0\parallel_{g,\widetilde{h}} + \parallel \n_{f_0} \parallel_{\widetilde{h}} )\\
            \leq & a_n + 2 M(\alpha_{max}) c \parallel g_n - g_{n-1} \parallel^{\frac{1}{2}}_{g,\widetilde{h}}  (2\widetilde{K})^{n} ( \parallel df_0\parallel_{g,\widetilde{h}}+ \parallel \n_{f_0} \parallel_{\widetilde{h}} )
          \end{align*}
          the last inequality comes from combining the first condition with the triangle inequality:
          \begin{align*}
              \parallel f_{n-1}^{*}h-g_{n} \parallel^{\frac{1}{2}}_{g,\widetilde{h}}  \leq&  \parallel f_{n-1}^{*}h-g_{n -1 } \parallel^{\frac{1}{2}}_{g,\widetilde{h}}  + \parallel g_{n-1}-g_n \parallel^{\frac{1}{2}}_{g,\widetilde{h}}  \\ 
              \leq & 2 \parallel g_n - g_{n-1} \parallel_{g,\widetilde{h}}^{\frac{1}{2} }
          \end{align*}
          We take the constant $T:= 2 M(\alpha_{max}) c ( \parallel df_0\parallel_{g,\widetilde{h}}+ \parallel \n_{f_0} \parallel_{\widetilde{h}} )$.
\end{itemize}
By construction, the sequence $(f_n)_{n\in\N}$ will converge to a $C^{1}$ embedding $f_\infty$ which is isometric between $(C,g)$ and $(W,h)$. Indeed, condition $(2)$ implies that the sequence $(f_n)_{n\in\N}$ converges to a continuous map $f_\infty$ and such that $$\parallel f_0-f_{\infty}\parallel_{\widetilde{h}} \leq \varepsilon$$
Moreover, combining condition $(3)$ and the fact that 
$$\sum \parallel g_n - g_{n-1}\parallel^{\frac{1}{2}}_{g,\widetilde{h}} (2\widetilde{K})^{n} = \parallel \Delta \parallel_{g,\widetilde{h}}^\frac{1}{2} \sum \sqrt{\delta_{n-1}-\delta_{n}}(2\widetilde{K})^{n}< \infty$$
  ensures that $f_\infty$ is $C^{1}$. We conclude by condition $(1)$ that $f_\infty^{*}h=g$. 
  \section{The Global construction}
  The constructions presented in section $3$ are defined locally and can be extended globally by the relative property of the corrugation process. We give here the global construction for completeness. 
    \begin{proposition}
        Let $(V,g)$ be a closed Riemannian surface and $(W,h)$ a Lorentzian manifold. For any long spacelike embedding $f:V\rightarrow W$ (i.e, $\Delta =f^*h-g $ is semi-definite positive), and any $\varepsilon>0$, there exists a spacelike embedding $F : V\rightarrow W$ such that:
$$\left\|F^{*}h-g\right\|_{C^{0}}\leq \varepsilon$$ 
    \end{proposition}
\begin{proof}

The construction follows the approach of Nash \cite{nash1954c1} and relies on the relative property of Theilllière's corrugation process \cite{theilliere2022convex}: 
\begin{itemize}
\item We choose a finite set $\phi_i: U_i \subset \mathbb{R}^2 \rightarrow V_i \subset V$ of local parametrizations of $V$ and a partition of unity $\{\rho_i\}$ subordinated to it ($supp(\rho_i) \subset U_i$) for all $i\in [\![1,m]\!]$.
    \item We identify the space of inner products of $\mathbb{R}^{2}$ (given by symmetric positive definite matrices $\begin{bmatrix} E & F \\ F & G \\ \end{bmatrix}$) with the interior of the cone $S^{+}(\mathbb{R}^{2}) \subset \R^{3}$ defined by the conditions:
    $$
\begin{cases} EG-F^{2} > 0, \\ E, G > 0 \end{cases}
$$
\item We note $\Delta_i =:\phi_i^*\Delta_{|\phi_i(supp(\rho_i))}$ so that we have
$$ \Delta = \sum_i \rho_i \Delta_i.$$
Since $supp(\rho_i)$ is compact and $\Delta_i : supp(\rho_i) \rightarrow S^{+}(\mathbb{R}^{2})$ is continuous, we can find $p_i$ squares of linear forms $\ell_{i,1}\otimes \ell_{i,1} , \ell_{i,2}\otimes \ell_{i,2}, ..., \ell_{i,p_i}\otimes \ell_{i,p_i}$ such that 
$$\Delta_i = \sum_{j=1}^{{p_i}} \eta_{i,j} \ell_{i,j} \otimes \ell_{i,j}$$ 
where each $\eta_{i,j} : supp(\rho_i) \rightarrow \mathbb{R}^{+} $ is a smooth function. For a proof of this fact, check lemma $29.3.1$ in \cite{eliash}.
\item Working chart by chart, we will build a sequence of intermediary maps $f=F_0, F_1, ..., F_m$ defined on $V$ such that
$$\left\|F_i^{*}h-(f^{*}h - \sum_{k=1}^{i}\rho_j \Delta_j)\right\|_{C^{0}}\leq \varepsilon$$
The map $F_m$ will be $\varepsilon$-isometric for the metric $g$ provided the corrugation numbers are large enough in each step.
 
\item Each map $F_i$ is built by applying $p_i$ successive corrugation process on $F_{i-1}$. We will use proposition \ref{local} in the chart $U_i$ to build a sequence of maps $F_{i-1}=F_i^{0}, F_i^{1},..., F_i^{p_i}=F_i $, where 
$$F_i^{j}=CP(F_i^{j-1}, \mu_i^{j}, N_{i,j})$$ 
and $\mu_i^{j} = F_{i}^{j-1 *}h - \rho_i\eta_{i,j} \ell_{i,j} \otimes \ell_{i,j}$. By remark \ref{gluing}, the maps are well defined on $\Sigma$, in fact, they all coincide outside $\phi_i(supp(\rho_i))$ since the metric $\mu_i^{j}$ is equal to $F_{i-1}^*h$ there.
\end{itemize}

Even though we introduced the construction when $W$ is the Minkowski space, the construction is still well defined for general $W$ and the conclusion remains the same. Indeed, in the coordinate free definition \ref{defcorrug} of the corrugation process, If we take $U$ to be compact and such that $f^*TW$ is trivial, then we can find a smaller neighborhood of the zero section of $f^*TW \rightarrow U$ where the exponential map is well defined. Since the loop family in definition \ref{defcorrug} $\Gamma(x,s)$ is bounded, we can choose $N$ large enough so that $\frac{1}{N}\Gamma(x,N\pi(x))$ lies inside this neighborhood. Using the same loop family \ref{loop} as in proposition \ref{local}, we get an $\varepsilon$-isometric embedding.
   \medskip
   
    Using standard arguments of Nash-Kuiper, the construction above can be adapted to prevent self intersections provided, we start with an embedding. Hence, it produces embeddings and not only immersions. Moreover, if the Euler characteristic of $V$ is different than $0$, the immersions constructed are embeddings. Indeed, we can choose a tubular neighborhood of $f(V)$ which is globally hyperbolic, of the form $V\times ]-\varepsilon,\varepsilon[$, where $f(V)$ is identified
with level $0$. By $C^{0}$-closeness, we can construct the immersion $F$ to lie inside $V\times ]-\varepsilon,\varepsilon[$. Now, notice that $F$ composed with the projection on the first factor is a covering map (since $F$ is spacelike), and since $V$ is compact and its Euler characteristic is different from $0$, we conclude that it is a diffeomorphism. Therefore, $F$ is an embedding.
\end{proof}

\section{Appendix}
In this appendix, we prove that the function $\psi$ appearing in the proof of lemma \ref{lem} and defined by:
$$\psi : \alpha \in [0,+\infty[ \mapsto \frac{ \frac{\sqrt{2\cosh(\al)^{2} -2\varphi(\al)} +\sinh(\al)} {\varphi(\alpha) }}{\sqrt{1-\frac{1}{\varphi(\alpha)^{2}}}}$$
is well defined and continuous at zero. 
\\ First, notice that: 
\begin{align*}
    \psi(\al) =& \frac{ \sqrt{2\cosh(\al)^{2} -2\varphi(\al)} +\sinh(\al)} {\sqrt{\varphi(\al)^{2}-1}}
\end{align*}
We note 
$$\psi_1(\al):=\frac{\cosh(\al)^{2} -\varphi(\al)}{\varphi(\al)^{2}-1} \;\;\; \text{and} \;\; 
\psi_2(\al):= \frac{\sinh(\al)^{2}}{\varphi(\al)^{2}-1} $$
In the following, we will use L'Hôpital's rule to prove that $\psi_1$ and $\psi_2$ are well defined at $0$. Noting that $\psi=\sqrt{2\psi_1}+\sqrt{\psi_2}$, we can conclude.

By \ref{fi}, we have:
$$\varphi^{\prime}(\al)=\int_{0}^{1} \cos (2 \pi s) \sinh (\alpha \cos (2 \pi s)) \mathrm{d} s$$
$$\varphi^{\prime\prime}(\al)=\int_{0}^{1} \cos (2 \pi s)^{2} \cosh(\alpha \cos (2 \pi s)) \mathrm{d} s$$

Since $\cosh(0)^{2}-\varphi(0)=0=\varphi(0)^{2}-1$ and both $\cosh(\al)^{2} -\varphi(\al)$ and $\varphi(\al)^{2}-1$ are differentiable on the right, and $(\varphi(\al)^{2}-1)^\prime= 2\varphi^{\prime}(\al)\varphi(\al)$ is strictly positive near $0$, we can apply L'Hôpital's rule: 
\begin{align*}
    \lim\limits_{\al\to 0} \psi_1(\al) =& \lim\limits_{\al\to 0}\frac{ (\cosh(\al)^{2} -\varphi(\al))^{\prime} }{2\varphi^{\prime}(\al)\varphi(\al)} \\
     =& \lim\limits_{\al\to 0} \frac{ \sinh(2\al)-\varphi^{\prime}(\al) }{ 2\varphi^{\prime}(\al)\varphi(\al)} \\
\end{align*}
Since $\sinh(0)-\varphi^{\prime}(0))=0=2\varphi^{\prime}(0)\varphi(0)$, we need to apply L'Hôpital's rule a second time, we get:
\begin{align*}
    \lim\limits_{\al\to 0} \psi_1(\al)   =& \lim\limits_{\al\to 0} \frac{ (\sinh(2\al)-\varphi^{\prime}(\al))^{\prime }}{ (2\varphi^{\prime}(\al)\varphi(\al))^{\prime}}\\
     = & \lim\limits_{\al\to 0} \frac{ 2\cosh(2\al)-\varphi^{\prime \prime}(\al) }{ 2\varphi^{\prime \prime}(\al)\varphi(\al) + 2\varphi^{\prime}(\al)^{2}}
\end{align*}
We have $$2\varphi^{\prime \prime}(0)\varphi(0) + 2\varphi^{\prime}(0)^{2}=2 \int_{0}^{1} \cos (2 \pi s)^{2} \mathrm{d} s =1 $$
Hence $\psi_1(0)$ is well defined. 
\medskip

 By the same arguments, we have: 
\begin{align*}
    \lim\limits_{\al\to 0} \psi_2(\al) =& \lim\limits_{\al\to 0}\frac{ \sinh(\al)^{2}  }{ 2\varphi^{\prime}(\al)\varphi(\al)} \\
     =& \lim\limits_{\al\to 0} \frac{ \sinh(2\al) }{ 2\varphi^{\prime}(\al)\varphi(\al)} \\
     =& \lim\limits_{\al\to 0}\frac{ \sinh^{\prime }(2\al)}{ (2\varphi^{\prime}(\al)\varphi(\al))^{\prime}}\\
     = & \lim\limits_{\al\to 0}\frac{ 2\cosh(2\al) }{ 2\varphi^{\prime \prime}(\al)\varphi(\al) + 2\varphi^{\prime}(\al)^{2}}\\
     =& \frac{2}{2 \int_{0}^{1} \cos (2 \pi s)^{2} \mathrm{d} s} \\
     =& 2
\end{align*}
 Therefore, the function $\psi$ is well defined and continuous at $0$.

\bibliographystyle{plain}
%\printbibliography{bibliography.bib}
\bibliography{Bib.bib}
  \end{document}